\documentclass[11pt]{article}
\usepackage{amsthm}
\usepackage{amsmath}
\usepackage{amssymb}
\usepackage{enumerate}
\usepackage{epsfig}
\usepackage{paralist}
\usepackage{lpic}
\usepackage{hyperref}
\usepackage{color}

\newcommand{\C}{\mathcal{C}}

\newcommand{\R}{\mathbb{R}}
\newcommand{\N}{\mathbb{N}}
\newcommand{\Z}{\mathbb{Z}}
\newcommand{\ve}{\varepsilon}
\newcommand{\vp}{\varphi}

\newtheorem{Theorem}{Theorem}
\newtheorem{Proposition}[Theorem]{Proposition}
\newtheorem{Lemma}[Theorem]{Lemma}
\newtheorem{Corollary}[Theorem]{Corollary}
\theoremstyle{remark}
\newtheorem{remark}[Theorem]{Remark}
\theoremstyle{definition}
\newtheorem{Definition}[Theorem]{Definition}

\title{
Singular solutions for a class of traveling wave equations arising
in hydrodynamics \footnote{{\bf Acknowledgements.} The first author
is  supported by the FWF project J3452 ``Dynamical Systems Methods
in Hydrodynamics'' of the Austrian Science Fund. The second author
is partially supported by Ministry of Economy and Competitiveness of
the Spanish Government through grant DPI2011-25822 and grant
2014-SGR-859 from AGAUR, Generalitat de Catalunya.}}

\author{Anna Geyer$^{(1)}$and V\'{\i}ctor Ma\~{n}osa $^{(2)}$
  \\*[.1truecm]
{\small \textsl{$^{(1)}$ Dept. de Matem\`{a}tiques, Facultat de
Ci\`{e}ncies,}}
\\ {\small \textsl{Universitat Aut\`{o}noma de Barcelona,}}
\\ {\small \textsl{08193 Bellaterra, Barcelona, Spain}}
\\ {\small \textsl{annageyer@mat.uab.cat}}\\
\\ {\small \textsl{$^{(2)}$ Dept. de Matem\`{a}tica Aplicada III,}}
\\ {\small \textsl{Control, Dynamics and Applications Group (CoDALab)}}
\\ {\small \textsl{Universitat Polit\`{e}cnica de Catalunya}}
\\ {\small \textsl{Colom 1, 08222 Terrassa, Spain}}
\\ {\small \textsl{victor.manosa@upc.edu}}}
\begin{document}

\maketitle
\begin{abstract}

We give an exhaustive characterization of singular weak solutions for ordinary
differential equations of the form $\ddot{u}\,u +
\frac{1}{2}\dot{u}^2 + F'(u) =0$, where $F$ is an analytic function.
Our motivation stems from the fact that  in the context of hydrodynamics several
prominent equations are reducible to an equation of this form
upon passing to a moving frame. We construct peaked and cusped waves,
fronts with finite-time decay and compact solitary waves. We prove
that one cannot obtain peaked and compactly supported traveling waves for the
same equation. In particular, a peaked traveling wave cannot have compact
support and vice versa. To exemplify the approach we apply our
results to the Camassa-Holm equation and the equation for surface waves
of moderate amplitude, and show how the different types of singular solutions
can be obtained varying the energy level of the corresponding planar Hamiltonian systems.
\end{abstract}



\section{Introduction}\label{S_intro}

In the present paper we propose to study certain types of weak solutions for ordinary differential equations (ODE) of the form
\begin{equation}\label{E_TWeq}
 \ddot{u}\,u + \frac{1}{2}\dot{u}^2 + F'(u) =0,
\end{equation}
where $F$ is an analytic function. Our motivation stems from the
fact that a variety of model equations arising in the context of
hydrodynamics, among them the well-known Camassa--Holm equation
(cf.~\cite{Camassa1993,Camassa94,Fuchssteiner1981}) and the related equation for surface waves of
moderate amplitude (cf.~\cite{ConLan09,DurukMutlubas2013a,DurukMutlubas2013b,DGM14,Gasull2014}), are reducible to an ODE of the
form \eqref{E_TWeq} upon passing to a moving frame. Owing to the fact that every
solution of equation \eqref{E_TWeq} may be interpreted as a
traveling wave of a suitable underlying partial differential equation (PDE) we will call the solutions of \eqref{E_TWeq}
traveling waves.

The singular nature of Equation \eqref{E_TWeq} accounts for the
non-uniqueness of certain solutions, which we call \emph{singular solutions}.
These are in general weak solutions, but have stronger regularity
than one would expect a priori: the solutions are analytic except
for a countable number of points at which the equation is satisfied
in the limit. Furthermore, equation \eqref{E_TWeq} admits an order
reduction which allows us to see that under certain conditions on
$F$, the solutions are actually classical solutions of this reduced
equation.

The main result of this paper consists in the exhaustive
characterization of singular solutions of \eqref{E_TWeq} from
qualitative properties of the function $F(u)$. We show  that
equation \eqref{E_TWeq} admits solutions with peaks and cusps,
fronts  with finite-time decay and solitary solutions with compact support. Furthermore, we
find that one cannot obtain peaked and compactly supported solutions
for the same $F$. In particular, a peaked solution cannot have
compact support and vice versa. The characterization of classical
solutions of \eqref{E_TWeq} will be covered only very briefly for
the convenience of the reader, since our main focus lies in the
analysis of singular solutions.

We apply our results to the aforementioned nonlinear partial
differential equations, and show how the different types of singular
solutions are obtained varying the energy levels of the Hamiltonian
planar differential system corresponding to \eqref{E_TWeq}. It lies
beyond the scope of this paper to prove  in full generality that
every weak solution of \eqref{E_TWeq} is also a weak traveling wave
solution of an underlying PDE. For a discussion of this problem  we
refer the reader to  \cite{Ehrnstrom2009a} and \cite{Lenells2005a},
where it is shown that in the special case of the Camassa-Holm
equation every weak solution of \eqref{E_TWeq} is a weak traveling
wave solution of the underlying PDE. Following similar steps the
same result can be shown for the equation of surface waves of
moderate amplitude.

The structure of the paper is as follows. In
Section~\ref{S_prelimdef}, we give  the precise definitions of weak
and singular solutions and provide a preliminary result on the
non-uniqueness of solutions of \eqref{E_TWeq}. In Section \ref{S_TW}
we introduce the notion of \emph{elementary forms}, classical
solutions of \eqref{E_TWeq} defined on a subset of $\R$, from which
we construct singular solutions. Furthermore, we discuss how the
qualitative features of any traveling wave solution can be obtained
from the properties of $F$. The main results of the paper,
Propositions \ref{P_peaked}, \ref{P_cpsupp}, \ref{P_c1} and
\ref{P_cusped} are presented in Section \ref{S_STW} which is devoted
to the complete characterization of singular solutions. In Section
\ref{S_bif_approach}, we characterize the classical and singular
traveling waves of the Camassa-Holm equation and the equation for
surface waves of moderate amplitude in shallow water.

\section{Weak and  singular solutions}
\label{S_prelimdef}

Our focus lies in the
characterization of solutions which are not classical, so we require a weak
formulation of \eqref{E_TWeq}. Keeping in mind that any solution of
\eqref{E_TWeq} can be interpreted as a traveling wave of an underlying
PDE, we will consider only bounded solutions.
\begin{Definition}\label{D_TWS} We say that a bounded function $u\in H^1_{loc}(\R)$ is a \emph{traveling wave solution (TWS)} if it satisfies \eqref{E_TWeq} in the sense of distributions, i.e.~if $u$ satisfies
\begin{equation}\label{E_weak}
  \int_{\R} (u^2)_t \phi_t + (u_t)^2 \phi - 2F'(u) \phi \, dt = 0,
\end{equation}
for any test function $\phi \in \C^{\infty}_c(\R)$. We say that $u$
is a \emph{strong TWS} if it satisfies \eqref{E_TWeq} in the
classical sense.
\end{Definition}

It turns out that the concept of weak solutions is quite crude. Indeed, if
no further conditions are imposed it is possible to find a plethora
of weak solutions of \eqref{E_TWeq} giving rise to TWS with very
complex shapes. For instance it is known that the Camassa--Holm
equation can have TWS of the form $u=\vp(t)$ such that some of its
level sets $\{\vp(t)=k\}$ are cantor sets, cf.~\cite{Lenells2005a}. In the
present paper, however, we are interested in those solutions  which fail
to be strong TWS because of the singularity, but
which still have a certain degree of regularity, and in fact are
strong solutions of \eqref{E_TWeq} except for a finite or a
countable set of points (such solutions are, for instance,
peaked or cusped waves, fronts with finite time decay, compact solitary waves and
composite waves, see Section \ref{S_STW} for precise definitions). To
see this, observe that equation \eqref{E_TWeq} admits the order
 reduction
\begin{equation*}
 \frac{d}{dt}\Big(\frac{u \dot{ u}^2}{2} + F(u) \Big) = \dot u \,\Big(  \ddot{u}\,u + \frac{1}{2}\dot{u}^2 + F'(u)\Big),
\end{equation*}
which is equivalent to the fact that the planar differential system associated to
\eqref{E_TWeq} has a first integral.
Therefore, any classical solution of  \eqref{E_TWeq} naturally satisfies
\begin{equation}
 \label{E_singTW}
    \frac{u \dot u^2}{2} + F(u)=h, \, \mbox{ for some constant  } h\in\R.
\end{equation}
The singularity of \eqref{E_singTW} leads to the existence of non
constant solutions $u\in H^1_{loc}(\R)$ of \eqref{E_weak} which
satisfy \eqref{E_singTW} except, perhaps, at a countable number of
points where the derivative is not defined but the equation is
still satisfied in the limit. This motivates the
following definition:

\begin{Definition}\label{D_singularTWS} Let $u(t)\in H^1_{loc}(\R)$ be
a non constant TWS of  \eqref{E_TWeq}. Then $u(t)$ is called a
\begin{enumerate}[(a)]
\item  \emph{strong singular} TWS of
\eqref{E_TWeq} if $u$ is a  classical solution of  \eqref{E_singTW}  on $\R$.

\item  \emph{weak singular} TWS  of \eqref{E_TWeq} if $u$ is a classical solution of \eqref{E_singTW} on  $\R\setminus \mathcal{S}$, where $\mathcal{S}$ is the set of countably many points $t_k$  such that
\begin{equation}\label{E_weaksingular}
   \lim_{t\rightarrow t_k} \frac{u(t) \dot u(t)^2}{2} + F(u(t))=h.
\end{equation}
\end{enumerate}
\end{Definition}

For the remainder of this paper we set $h_0:=F(0)$. The next result establishes the non-uniqueness of
solutions of \eqref{E_singTW} for $h=h_0$. This will play a role in
the construction of singular TWS of \eqref{E_TWeq}, see Section \ref{S_STW}.

\begin{Lemma}\label{L_uniq2}
Consider equation \eqref{E_singTW} with $F$ an analytic function
such that $u\left(h-F(u)\right)>0$ for $
u\in\mathcal{U}\setminus\{0\}$ where $\mathcal{U}$ is a neighborhood
of $0$.
\begin{enumerate}[(a)]
  \item If $h\neq h_0$, then  equation \eqref{E_singTW} is not defined at
  $u=0$.
  \item If $h=h_0$ then  $u\equiv 0$ is a solution of equation
  \eqref{E_singTW}. Furthermore if $F'(0)\neq 0$ or $F'(0)=0$ and
  $F''(0)\neq 0$ then equation  \eqref{E_singTW} is not Lipschitz continuous in $u=0$.
\end{enumerate}
\end{Lemma}

\begin{proof}
 Notice that equation \eqref{E_singTW} may be written as
 \begin{equation*}
    \dot{u}=\pm\sqrt{2\,\frac{h-F(u)}{u}} =: v_h^{\pm}(u).
 \end{equation*}
Observe that if $h\neq h_0$ then $\lim\limits_{u\rightarrow 0}
v_h^{\pm}(u) =\pm \infty$, which proves statement (a). When $h=h_0$,
$u\equiv0$ is always a solution of \eqref{E_singTW}. Note that
$$
\frac{|v_{h_0}^{\pm}(u)
|}{|u|}=\sqrt{-2\left(\frac{F'(0)}{u^2}+\frac{F''(0)}{2!\,
u}+\frac{F^{(3)}(0)}{3!}+o(u)\right)}.
$$
Hence, if $F'(0)\neq 0$ or $F'(0)=0$ and $F''(0)\neq 0$ then
$\lim\limits_{u\rightarrow 0} |v_{h_0}^{\pm}(u)|/|u|=+\infty$.
Therefore it is not possible to find a constant $L>0$ such that
$|v_{h_0}^{\pm}(u)| \leq L |u|$ for $u \in \mathcal{U}$, and hence
the right-hand side of the differential equation fails to be
Lipschitz continuous in $u=0$. ~\end{proof}

\begin{remark}
Observe that $v_{h_0}^{\pm}(u)$ is Lipschitz continuous in $u=0$
only if $F'(0)=F''(0)=0$ and $F^{(3)}(0)\neq 0$. However, as a
consequence of Propositions \ref{P_peaked}, \ref{P_cpsupp} and
\ref{P_c1}, such $F$ do not yield any singular solutions, and
therefore they will not be considered in this paper.
\end{remark}

\section{Traveling wave solutions from qualitative properties of $F$}\label{S_TW}

The planar system associated to equation \eqref{E_TWeq} is given by
  \begin{equation}
  \label{E_TWSys}
       \left\{
      \begin{array}{l l }
    \dot u =v\vspace{0.8em}\\
    u\,\dot v = - F'(u) - \frac{1}{2}\,v^2.
      \end{array}\right.
  \end{equation}
A straightforward computation shows that system \eqref{E_TWSys}
possesses the first integral
  \begin{equation}
   \label{Hamiltonian}
     H(u,v) = \frac{u\,v^2}{2}+ F(u),
  \end{equation}
whose \emph{energy} $H(u,v)=h$ is therefore constant along
solutions. We will study how the TWS of  \eqref{E_TWeq} can be
obtained from the orbits of the associated system \eqref{E_TWSys},
and how these orbits depend on the qualitative properties
 of $F$. Our study resembles the study of
conservative systems of one degree of freedom \cite[$\S$12]{Arnold1992}. For
such systems, the qualitative features of the phase portrait are in
correspondence with the ones of the potential function. Analogously,
we find that the qualitative features of the function  $F$
characterize the phase portrait of \eqref{E_TWSys} and therefore the
types of TWS that  \eqref{E_TWeq} may exhibit.
In contrast to potential systems, the Hamiltonian \eqref{Hamiltonian}  has no purely kinetic term, which leads to the presence of a singular line  in the phase portrait and the existence of solutions defined on subsets of $\R$. Without loss of
generality we consider the system only in the half-plane
$(\{u>0\}\cup\{u=0\})\times \R$, which is not  restrictive since our
analysis is of a qualitative nature. It is worth mentioning however
that in some applications it might be interesting to consider
solutions in $\{u\leq 0\}$.

The solutions of equation \eqref{E_TWeq} are associated to the orbits of \eqref{E_TWSys}, which correspond to the level sets of the energy $H(u,v)=h$, i.e.~they lie on
 curves of the form $\gamma_h=\{H=h\}$ which are composed of two symmetric branches $(u,v_h^{\pm}(u))$, where
\begin{equation}
  \label{E_vh}
   v_h^{\pm}(u)  = \pm \sqrt{2\, \frac{h-F(u)}{u}}.
\end{equation}
Therefore, solutions of equation \eqref{E_TWeq} corresponding to
energy $h$ exist if there is a non-empty set $I\subseteq \{u\geq
0\}$ such that $F(u)-h<0$ for all $u\in I$. Recall that our interest
lies in the analysis of bounded solutions, which correspond to
branches of the curves $\gamma_h$ that are bounded in the
$u$-direction. These branches are defined if and only if either
there exist $m_1,m_2>0$ such that $F(m_1)=F(m_2)=h$ and $F(u)<h$ for
all $u\in(m_1,m_2)$, or there exists $m>0$ such that $F(m)=h$ and
$F(u)<h$ for all $u\in(0,m)$. The latter case
corresponds to a type of strong solutions of \eqref{E_TWeq} whose
maximal domain of definition is a (finite) subset of $\R$. We call these solutions  \emph{elementary forms}, which
 play an important role in our construction of singular TWS.
 The former case corresponds to
classical smooth solutions which are uniquely defined on $\R$.
Observe that possible oscillations of $F$ in $(m_1,m_2)$ or $(0,m)$ affect the solutions of \eqref{E_TWeq} only at the level of their convexity.

\subsection{Smooth TWS on $\mathbb{R}$}\label{S_SmoothTW}

In this subsection we briefly summarize how smooth TWS are characterized in terms of  qualitative
aspects of $F$.
Assume that there exist $m_1,m_2>0$ such that $F(m_1)=F(m_2)=h_m$
and $F(u)<h_m$ for all $u\in(m_1,m_2)$. Then the two branches of
$\gamma_{h_m}$ given by $(u,v_{h_m}^\pm(u))$ are defined for
$u\in[m_1,m_2]$ and coincide at the points $p_1=(m_1,0)$ and
$p_2=(m_2,0)$.

Observe that the critical points of system \eqref{E_TWSys} in
$\{u>0\}\times \R$ are  of the form $(u,0)$ where $F'(u)=0$. Hence a
point $p_i=(m_i,0)$ will be a critical point if $F'(m_i)=0$ and a
regular point otherwise. We distinguish between the following
cases (see Figure \ref{Fig_smooth}):

\begin{enumerate}
  \item[(A)] If $F'(m_i)\neq 0$ for $i=1,2$,
 then the points $p_1$ and $p_2$ are regular, and the two branches
$(u,v_{h_m}^\pm(u))$ give rise to an isolated closed curve, which is
a periodic orbit of \eqref{E_TWSys}. This periodic orbit corresponds
to a strong smooth periodic solution $u(t)$ of equation
\eqref{E_TWeq}. Observe that since  $v_{h_m}^+(u)$ and $v_{h_m}^-(u)$ are
not multivalued functions, the closed curve has no lobes and
therefore $u(t)$ reaches a unique local minimum at
$m_1$ and a unique local maximum at $m_2$ in each period. Furthermore, this periodic solution is symmetric with respect to its local minima and maxima.

  \item[(B)] If $F'(m_1)= 0$ and $F'(m_2)\neq 0$  then $p_1$ is a
  critical point and $p_2$ is a regular one. In this case, the branches
$(u,v_{h_m}^\pm(u))$ yield a homoclinic loop giving rise to  a
smooth solitary wave solution of \eqref{E_TWeq} which has a unique
maximum at $m_2$ and which decays exponentially on either side of the maximum such that $\lim_{t\rightarrow\pm
\infty}u(t)=m_1$ (recall that since $p_1$ is a critical point, the
time that the orbit of \eqref{E_TWSys} takes to leave from or to
reach this point is infinite). Furthermore, this solitary wave solution is symmetric with respect to its maximum. If $F'(m_1)\neq 0$ and $F'(m_2)=0$
then there appears a solitary wave of \eqref{E_TWeq} with a unique
minimum at $m_1$ and such that $\lim_{t\rightarrow\pm
\infty}u(t)=m_2$.

  \item[(C)] If $F'(m_1)= F'(m_2)=0$  then $p_1$ and
  $p_2$ are critical points, and the branches $(u,v_{h_m}^\pm(u))$  yield a heteroclinic loop connecting these points  in an
  infinite time. This pair of connecting orbits gives rise to a  smooth
  front decaying from $m_2$ to $m_1$ such that $\lim_{t\rightarrow-\infty}u(t)=m_2$ and
  $\lim_{t\rightarrow+\infty}u(t)=m_1$, and another smooth front increasing from  $m_1$ to $m_2$  such that $\lim_{t\rightarrow-\infty}u(t)=m_1$ and
  $\lim_{t\rightarrow+\infty}u(t)=m_2$.
\end{enumerate}

\begin{figure}
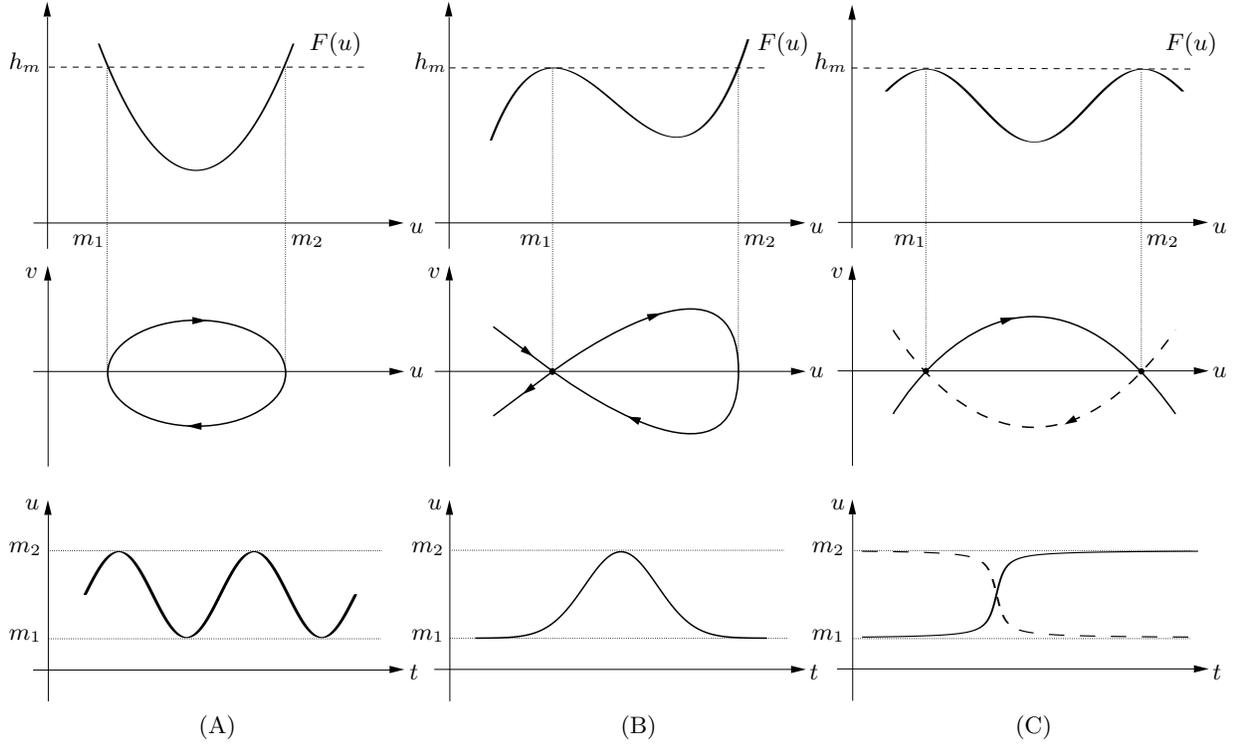

\footnotesize
  \centering
    \begin{lpic}[l(2mm),r(2mm),t(2mm),b(2mm)]{FigureAll1(0.35)}
      \lbl[l]{143,180;$u$}
      \lbl[l]{143,125;$u$}
      \lbl[l]{-3,163;$v$}
      \lbl[l]{143,11;$t$}
      \lbl[l]{105,250;$F(u)$}
      \lbl[c]{-3,243;$h_m$}

      \lbl[l]{15,175;$m_1$}
      \lbl[l]{98,175;$m_2$}

      \lbl[l]{-3,75;$u$}
      \lbl[c]{-3,58;$m_2$}
      \lbl[c]{-3,25;$m_1$}

      \lbl[c]{70,-10;(A)}

      \lbl[l]{295,180;$u$}
      \lbl[l]{295,125;$u$}
      \lbl[l]{150,163;$v$}
      \lbl[l]{295,11;$t$}
      \lbl[l]{275,250;$F(u)$}
      \lbl[c]{151,243;$h_m$}

      \lbl[l]{185,175;$m_1$}
      \lbl[l]{270,175;$m_2$}

      \lbl[l]{150,75;$u$}
      \lbl[c]{150,58;$m_2$}
      \lbl[c]{150,25;$m_1$}

      \lbl[c]{230,-10;(B)}

      \lbl[l]{448,180;$u$}
      \lbl[l]{448,125;$u$}
      \lbl[l]{303,163;$v$}
      \lbl[l]{448,11;$t$}
      \lbl[l]{430,250;$F(u)$}
      \lbl[c]{303,243;$h_m$}

      \lbl[l]{327,175;$m_1$}
      \lbl[l]{423,175;$m_2$}

      \lbl[l]{303,75;$u$}
      \lbl[c]{302,58;$m_2$}
      \lbl[c]{302,25;$m_1$}

      \lbl[c]{380,-10;(C)}
    \end{lpic}
  \caption{Global strong TWS of \eqref{E_TWeq}. (A) Periodic wave; (B)
Solitary wave or pulse; (C) Smooth front.}
\label{Fig_smooth}
\end{figure}

\subsection{Elementary forms: smooth TWS on a  subset of $\mathbb{R}$.}\label{S_ElemForms}

In this subsection we discuss two types of  elementary forms. The first type, studied in cases (a) and (b)
below,  are the ones associated to the curves $\gamma_{h_0}$ which
correspond to the energy level $h_0=F(0)$ and which intersect the
singular line $\{u=0\}\times \R$ at a finite point. The second type
are the ones associated to curves $\gamma_{h}$ that are unbounded in
the component $v$ and are arbitrarily close to the singular line
$\{u=0\}\times \R$ at infinity. They are studied in the case (c).
  In both cases, it will be a key step to study whether the orbits
of \eqref{E_TWSys} approach the singular line $u=0$ in finite time.
To this end we prove the following auxiliary result.

\begin{Lemma}\label{L_finite_time} Let
$F$ be an analytic function.  Suppose that there exist
$\varepsilon>0$ and $m>0$ such that $F(m)=F(0)$ and $F(u)<F(0)$ for
$u\in (0,\varepsilon)\cup (m-\varepsilon,m)$.
   \begin{enumerate}
    \item[(i)] Let $n\in \mathbb{N}$ be the lowest order such that $F^{(n)}(0)\neq 0$.
If $n\leq 2$ then \begin{equation}\label{E_int1} \int_0^\ve
\sqrt{\frac{u}{F(0)-F(u)}} du\end{equation} is convergent. Otherwise
the integral diverges.
    \item[(ii)] If $F'(m)>0$, then \begin{equation}\label{E_int2}\int_{m-\ve}^m  \sqrt{\frac{u}{F(0)-F(u)}}
    du\end{equation}
  is convergent. Otherwise the integral
diverges.
   \end{enumerate}
 \end{Lemma}

\begin{proof}
In both cases we will apply the following criterium:  let $f(x)$ be
an unbounded function at $x=c$ such that $f(x) \geq 0$ and
$\lim_{x\rightarrow c} f(x) |x-c|^k = A$, where $A \neq \infty$ and
$A\neq 0$. Then, for $k<1$ the integral $\int_a^c f(x) dx$ is
convergent, whereas for $k\geq 1$ it is divergent. To prove (i), set
\[
f(u):=\sqrt{\frac{u}{F(0)-F(u)}},\] and  $n\in\N$ such that
$F^{(n)}(0)<0$ and $F^{(k)}(0)=0$ for $k<n$. Hence
\[
    f(u)=\sqrt{\frac{-n!}{F^{(n)}(0) u^{n-1}+o(u^n)}},
 \]
 and
\[
    \lim_{u\rightarrow 0^+}f(u)|u|^{\frac{n-1}{2}}=\sqrt{\frac{-n!}{F^{(n)}(0)}}=:K, \mbox{ where }
0\neq K \neq \infty.
\]
Therefore, using the criterium stated above, for $n\leq 2$ the
integral \eqref{E_int1} converges and otherwise it diverges. To prove
(ii) let
\[
    f(u):=\sqrt{\frac{u}{F(m)-F(u)}},
\]
 and $n\geq 1$ such that $F^{(n)}(m)<0$ and $F^{(k)}(m)=0$ for $k<n$. Then
 \[
  f(u)=\sqrt{\frac{-n!\,u}{F^{(n)}(m) (u-m)^{n}+o((u-m)^{n+1})}},
  \]
and therefore
\[
    \lim_{u\rightarrow 0^+}
f(u)|u-m|^{\frac{n}{2}}=\sqrt{\frac{-n!\,m}{F^{(n)}(m)}}=:K, \mbox{
where } 0\neq K \neq \infty .
\]
Hence for $n=1$ the integral \eqref{E_int2} converges and otherwise
it diverges. ~\end{proof}

We start with a discussion of solutions  associated to
the energy level $h=h_0$. Suppose that there exists $m>0$ such
that $F(m)=h_0$ and $F(u)<h_0$ for all $u\in(0,m)$. In this case,
the curves $\gamma_{h_0}$ intersect the singular line $\{u=0\}$ at a finite point,
and the branches $(u,v_{h_0}^\pm(u))$ are defined for $u\in[0,m]$,
and coincide at the point $p=(m,0)$. As in the discussion for
smooth solutions in  Section \ref{S_SmoothTW} we consider the following cases:

\begin{figure}[ht]
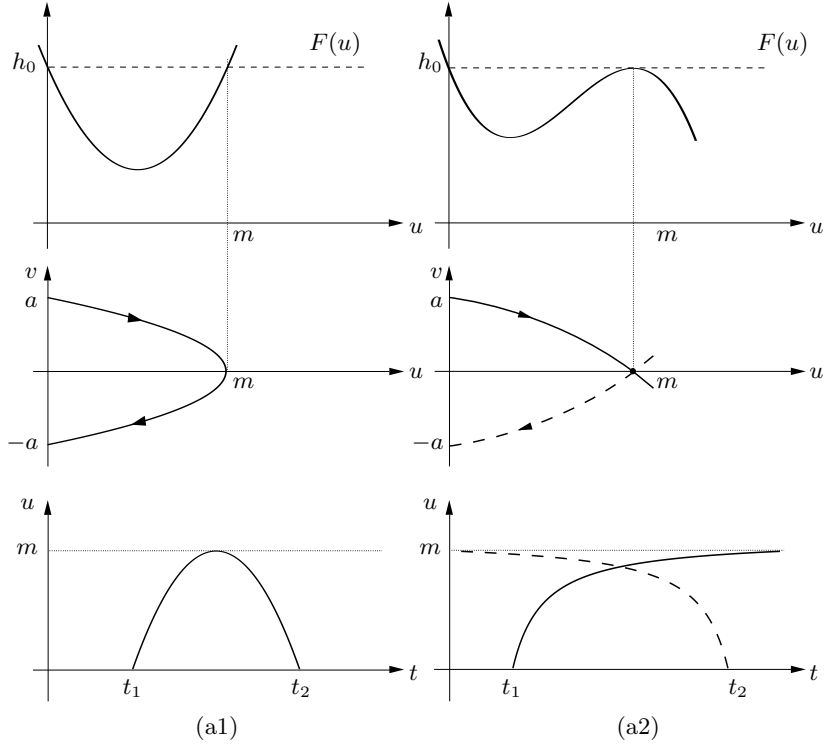

\footnotesize
  \centering
    \begin{lpic}[l(2mm),r(2mm),t(2mm),b(2mm)]{FigureAll2(0.35)}
      \lbl[l]{143,180;$u$}
      \lbl[l]{143,125;$u$}
      \lbl[l]{-3,163;$v$}
      \lbl[l]{143,11;$t$}
      \lbl[l]{105,250;$F(u)$}
      \lbl[c]{-3,243;$h_0$}
      \lbl[l]{-3,152;$a$}
      \lbl[l]{-10,98;$-a$}

      \lbl[l]{76,177;$m$}

       \lbl[l]{76,121;$m$}

       \lbl[c]{-2,75;$u$}
      \lbl[c]{-2,58;$m$}

      \lbl[c]{38,5;$t_1$}
      \lbl[c]{102,5;$t_2$}

       \lbl[c]{70,-10;(a1)}

      \lbl[l]{295,180;$u$}
      \lbl[l]{295,125;$u$}
      \lbl[l]{151,163;$v$}
      \lbl[l]{295,11;$t$}
      \lbl[l]{275,250;$F(u)$}
      \lbl[c]{151,243;$h_0$}
      \lbl[l]{151,152;$a$}
      \lbl[l]{144,98;$-a$}

      \lbl[l]{237,177;$m$}

      \lbl[l]{237,121;$m$}

      \lbl[c]{151,75;$u$}

      \lbl[c]{151,58;$m$}

      \lbl[c]{182,5;$t_1$}
      \lbl[c]{268,5;$t_2$}

      \lbl[c]{230,-10;(a2)}
    \end{lpic}
\caption{Strong solutions of \eqref{E_TWeq} defined in a subset of
$\R$ with energy $h=h_0$ when $F'(0)\neq0$.
These provide the elementary forms to construct peaked periodic and
solitary TWS.}
\label{Fig_sing-I}
\end{figure}

\begin{enumerate}
  \item[(a)] Suppose that $F'(0)\neq 0$. Then,   $$
\lim\limits_{u\rightarrow 0^+} v_{h_0}^\pm(u)=
\lim\limits_{u\rightarrow 0^+} \pm\sqrt{2\,\frac{F(0)-F(u)}{u}}=\pm
\sqrt{-2F'(0)} =:\pm a\neq \pm \infty.
$$ and therefore the branches $(u,v_{h_0}^\pm(u))$ intersect the
singular line $\{u=0\}$ at two distinct points $(0,\pm a)$. Moreover, they reach the singular line in finite time, since  Lemma \ref{L_finite_time} (i) guarantees that the
orbit $\gamma_{h_0}$ of system \eqref{E_TWSys} connects the point $(0,+a)$ with the
point $(\bar{u},v_{h_0}^+(\bar{u}))$ in \emph{finite time} $\Delta
T(\bar{u})$. Indeed, by direct integration of system \eqref{E_TWSys}
with $v=v^+_{h_0}(u)$, and recalling that $F'(0)<0$, we have
\begin{equation}\label{E_fin-time}
  \Delta T(\bar{u}) = \displaystyle{\int_0^{\bar{u}}
\frac{du}{v}}=\displaystyle{\int_0^{\bar{u}} \sqrt
\frac{u}{2\left(h_0-F(u)\right)}\,
du}=\displaystyle{\frac{1}{\sqrt{2}}\int_0^{\bar{u}} \sqrt
\frac{u}{F(0)-F(u)}\, du} <\infty.
\end{equation}

\begin{enumerate}
  \item[(a1)]   If $F'(m)\neq 0$, then $p$ is a regular point
  and there is an isolated orbit of \eqref{E_TWSys} connecting the
points $(0,+a)$ and $(0,-a)$, see Figure \ref{Fig_sing-I}\,(a1). By the
above arguments, this orbit connects these points in a finite time,
hence it gives rise to a strong solution $u(t)$ of \eqref{E_TWeq}
defined for some finite interval $(t_1,t_2)$. As in the smooth
cases, this isolated orbit defines a curve with no lobes and
therefore $u(t)$ has a unique maximum $m$ at $t=(t_1+t_2)/2$.  This
solution is symmetric with respect its maximum due to the  symmetry
of the curve $\gamma_{h_0}$.

  \item[(a2)] If $F'(m)= 0$, then $p$ is a critical point and
  the branch $(u,v_{h_0}^+(u))$ defines an orbit
  leaving its $\alpha$-limit   $(+a,0)$ in  finite time
 and reaching its $\omega$-limit $p$ in infinite time. This orbit corresponds to
  a strong solution $u_+(t)$ of \eqref{E_TWeq}
defined for some interval $(t_1,+\infty)$, $t_1\in\R$, such that
$\lim_{t\rightarrow t_1^+} u_+(t)=0$, $\lim_{t\rightarrow t_1^+}
\dot u_+(t)=a$ and $\lim_{t\rightarrow \infty} u_+(t)=m$, Figure
\ref{Fig_sing-I}\,(a2). Analogously, the  branch $(u,v_{h_0}^-(u))$
gives an orbit such that each point on the orbit connects with the
point $p$ in an infinite time, and  with $(-a,0)$ in a finite time.
This gives a strong solution $u_-(t)$ of \eqref{E_TWeq} defined for
some interval $(-\infty,t_2)$, $t_2\in\R$, such that
$\lim_{t\rightarrow t_2^-} u_-(t)=0$, $\lim_{t\rightarrow t_2^-}
\dot u_-(t)=-a$ and $\lim_{t\rightarrow -\infty} u_-(t)=m$. Notice
that $u_+(t_1-t)=u_-(t_1-t)$ when $t_1=t_2$ due to the symmetry of
system \eqref{E_TWSys}.
\end{enumerate}

\begin{figure}[ht]
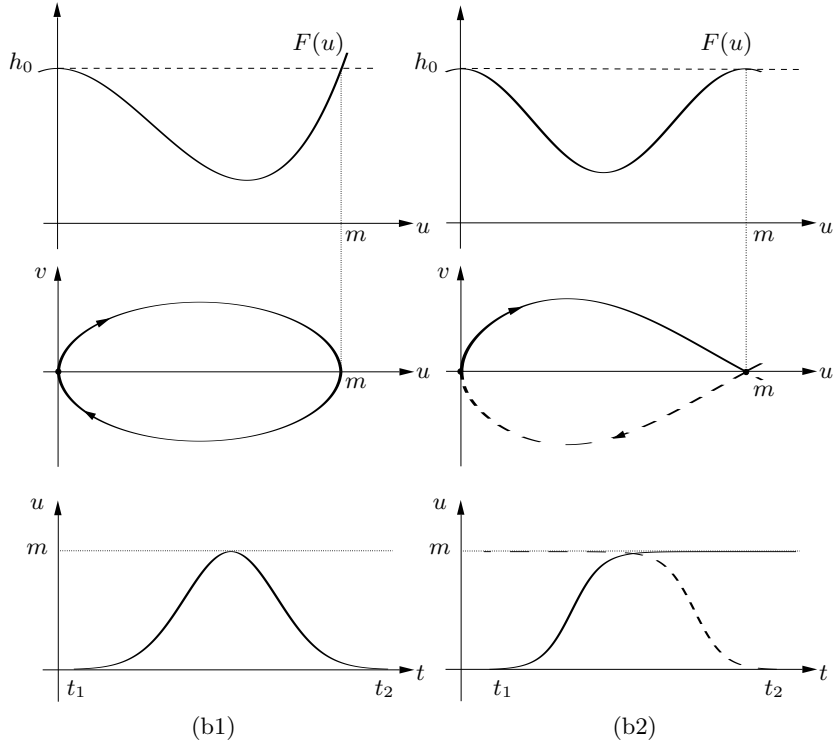

\footnotesize
  \centering
    \begin{lpic}[l(2mm),r(2mm),t(2mm),b(2mm)]{FigureAll3(0.35)}
      \lbl[l]{147,180;$u$}
      \lbl[l]{147,125;$u$}
      \lbl[l]{2,163;$v$}
      \lbl[l]{147,11;$t$}
      \lbl[l]{100,250;$F(u)$}
      \lbl[c]{-3,243;$h_0$}

      \lbl[l]{120,177;$m$}

       \lbl[l]{120,121;$m$}

      \lbl[c]{3,58;$m$}

      \lbl[c]{3,75;$u$}

      \lbl[c]{18,5;$t_1$}
      \lbl[c]{135,5;$t_2$}

       \lbl[c]{70,-10;(b1)}

      \lbl[l]{300,180;$u$}
      \lbl[l]{300,125;$u$}
      \lbl[l]{155,163;$v$}
      \lbl[l]{300,11;$t$}
      \lbl[l]{255,250;$F(u)$}
      \lbl[c]{151,243;$h_0$}

      \lbl[l]{275,177;$m$}

      \lbl[l]{275,118;$m$}

      \lbl[c]{156,58;$m$}
      \lbl[c]{156,75;$u$}

      \lbl[c]{180,5;$t_1$}
      \lbl[c]{283,5;$t_2$}

      \lbl[c]{230,-10;(b2)}
    \end{lpic}
 \caption{Strong solutions of \eqref{E_TWeq} defined in a subset of
$\R$ with energy $h=h_0$  when $F'(0)= 0$.
These provide elementary forms to construct TWS of class
$\mathcal{C}^1(\R)$ (with compact support).}\label{Fig_sing-II}
\end{figure}

  \item[(b)] Suppose that $F'(0)=0$. Then,
  $$
\lim\limits_{u\rightarrow 0^+} v_{h_0}^\pm(u)=\pm \sqrt{-2F'(0)} =0.
$$
Hence the branches $(u,v_{h_0}^\pm(u))$ are defined for $u\in[0,m]$,
and coincide in $(0,0)$ and $p=(m,0)$.

\begin{enumerate}
  \item[(b1)]   If $F'(m)\neq 0$, then $p$ is a regular point.
  Hence,  there is an isolated orbit of \eqref{E_TWSys} in $\{u>0\}\times \R$
  whose $\alpha$ and $\omega$--limit is $(0,0)$, Figure \ref{Fig_sing-II}\,(b1). Lemma \ref{L_finite_time} (i) and (ii) ensures
that the orbit of \eqref{E_TWSys} connects $(0,0)$ with $(m,0)$
 in finite time if $F''(0)<0$ and $F'(m)>0$.
  Therefore, it corresponds to a strong solution $u(t)$ of
\eqref{E_TWeq} defined in a finite interval $(t_1,t_2)$, such that
$\lim_{t\rightarrow t_1^+}u(t) = \lim_{t\rightarrow
t_2^-}u(t)=0$ and $\lim_{t\rightarrow t_1^+}\dot u(t)=$ $\lim_{t\rightarrow
t_2^-}\dot u(t)=0$. In addition, $u(t)$ reaches a unique local maximum
$m$ at $t=(t_1+t_2)/2$ and is symmetric with respect to this maximum due to the symmetry of system \eqref{E_TWSys}.

  \item[(b2)] If $F'(m)= 0$, then $p$ is a critical point and
there are two orbits with energy $h_0$ connecting $(0,0)$  with $p$, cf.~Figure~\ref{Fig_sing-II}\,(b2).
In view of Lemma \ref{L_finite_time}, any point of this orbit
connects with $(0,0)$ in finite time, thus the heteroclinic orbit
connecting $(0,0)$ with $p$ through the branch $(u,v_{h_0}^+(u))$
corresponds to a strong solution $u_+(t)$ of \eqref{E_TWeq} defined
in $(t_1,\infty)$, $t_1\in \R$,  such that $\lim_{t\rightarrow
t_1^+}u_+(t)=\lim_{t\rightarrow
t_1^+}\dot u_+(t)=0$ and $\lim_{t\rightarrow \infty}u(t)=m$.  The  branch $(u,v_{h_0}^-(u))$ corresponds to a strong solution
$u_-(t)$  defined in $(-\infty,t_2)$ with $t_2\in
\R$  such that $\lim_{t\rightarrow t_2^-}u_-(t)=\lim_{t\rightarrow
t_1^+}\dot u_-(t)=0$ and
$\lim_{t\rightarrow -\infty}u(t)=m$. Notice that
$u_+(t_1+t)=u_-(t_1-t)$ when $t_1=t_2$ due to the symmetry of system
\eqref{E_TWSys}.
\end{enumerate}
\end{enumerate}

The second type of solutions are the ones associated to curves
$\gamma_h$ which are defined for $u\in (0,m]$ but do not tend to a
finite point on the singular line $\{u=0\}\times \R$, hence they satisfy
$$
\lim\limits_{u\rightarrow 0^+} v_{h}^\pm=\pm\infty.
$$
Such curves appear if and only if there exists $m>0$ such that
$F(u)<h_m=F(m)$ for all $u\in[0,m)$, see Figure \ref{Fig_sing-III}.
 In addition,
observe that for $0<\bar{u}<m$
\begin{equation}\label{E_fin-time-2}
\Delta T(\bar{u})=\int_0^{\bar{u}}
\frac{du}{v}=\frac{1}{\sqrt{2}}\,\int_0^{\bar{u}} \sqrt
\frac{u}{F(m)-F(u)} du <\infty,
\end{equation}
since  $F(m)-F(u)>0$ for $u\in[0,m)$. Consequently,  the  orbits
 of \eqref{E_TWSys} passing through the points
 $(\bar{u},v_{h_m}^\pm(\bar{u}))$ are unbounded
in the component $v$ but aproach the singular line $u=0$ at infinity in
 finite time.
Analogously to the previous cases, we distinguish between two
scenarios:

\begin{enumerate}
\item[(c1)] If $F'(m)\neq 0$ then the branches $(u,v_{h_m}^\pm(u))$ coincide at
 the regular point $p=(m,0)$.
Thus there is an unbounded orbit of \eqref{E_TWSys} tending to
infinity in the $v$-direction in finite time, Figure
\ref{Fig_sing-III}\,(c1). It corresponds to a strong solution $u(t)$ of
\eqref{E_TWeq} defined in a finite interval $(t_1,t_2)$ such that
$\lim_{t\rightarrow t_1^+}u(t)=\lim_{t\rightarrow t_2^-}u(t)=0$ and
$\lim_{t\rightarrow t_1^+} \dot u(t)=+\infty$, $\lim_{t\rightarrow
t_2^-}\dot u(t)=-\infty$. This solution has a unique maximum $m$ at
$t=(t_1+t_2)/2$ and it is symmetric with respect this point.

\item[(c2)] If $F'(m)= 0$ then the branches $(u,v_{h_m}^\pm(u))$
 coincide at the critical point $p=(m,0)$ and there are
two unbounded orbits of \eqref{E_TWSys} tending to infinity in the
$v$-direction in finite time, cf.~Figure \ref{Fig_sing-III}\,(c2), but
any point in the orbit takes an infinite time to reach the singular
point $p$.
The orbit defined by the branch $(u,v_{h_m}^+(u))$  corresponds to
a strong solution $u_+(t)$ of \eqref{E_TWeq} defined for
$t\in(t_1,\infty)$, $t_1\in\R$, such that
$\lim_{t\rightarrow t_1^+}u_+(t)=0$, $\lim_{t\rightarrow
t_1^+}\dot u_+(t)=+\infty$, and $\lim_{t\rightarrow \infty}u_+(t)=m$. Similarly, the
orbit defined by the branch $(u,v_{h_m}^-(u))$ corresponds to a
strong solution $u_-(t)$ of \eqref{E_TWeq} defined for
$t\in(-\infty,t_2)$, $t_2\in\R$, such that
$\lim_{t\rightarrow -\infty}u_-(t)=m$ and $\lim_{t\rightarrow
t_2^-}u_-(t)=0$, $\lim_{t\rightarrow t_2^-}\dot u_-(t)=-\infty$.
 Notice that as before, $u_+(t_1+t)=u_-(t_1-t)$ when $t_1=t_2$ due to the symmetry of system \eqref{E_TWSys}.

\end{enumerate}

\begin{figure}[h]
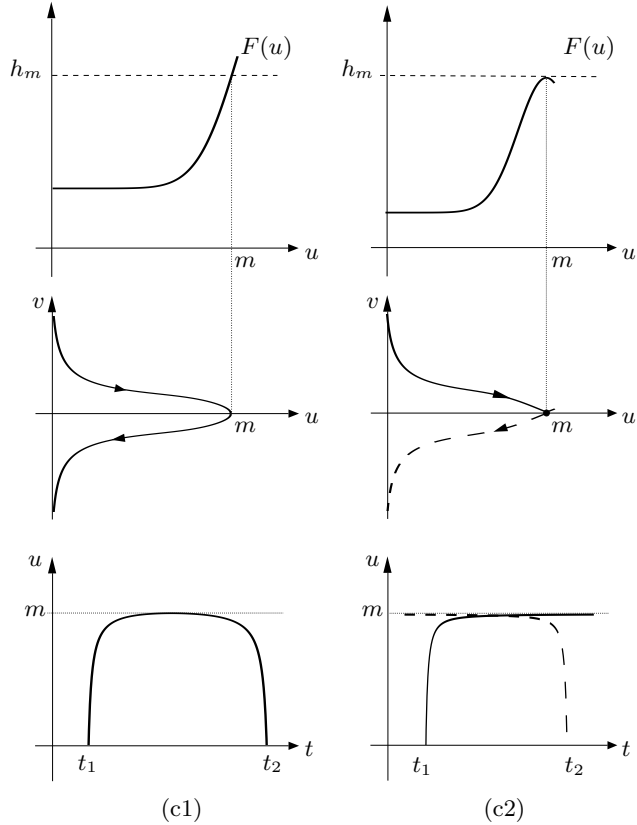

\footnotesize
  \centering
    \begin{lpic}[l(2mm),r(2mm),t(2mm),b(2mm)]{FigureAll4(0.39)}
      \lbl[l]{92,180;$u$}
      \lbl[l]{92,124;$u$}
      \lbl[l]{-1,163;$v$}
      \lbl[l]{92,11;$t$}
      \lbl[l]{70,250;$F(u)$}
      \lbl[c]{-3,243;$h_m$}

      \lbl[l]{68,177;$m$}

       \lbl[l]{68,121;$m$}

      \lbl[c]{0,58;$m$}
      \lbl[c]{0,75;$u$}

      \lbl[c]{18,6;$t_1$}
      \lbl[c]{81,6;$t_2$}

       \lbl[c]{50,-10;(c1)}

      \lbl[l]{200,180;$u$}
      \lbl[l]{200,124;$u$}
      \lbl[l]{113,163;$v$}
      \lbl[l]{200,11;$t$}
      \lbl[l]{180,250;$F(u)$}
      \lbl[c]{110,243;$h_m$}

      \lbl[l]{176,177;$m$}

      \lbl[l]{176,121;$m$}

      \lbl[c]{115,58;$m$}
      \lbl[c]{115,75;$u$}

      \lbl[c]{184,5;$t_2$}
      \lbl[c]{132,5;$t_1$}

      \lbl[c]{160,-10;(c2)}
    \end{lpic}
\caption{Strong solutions of \eqref{E_TWeq} defined in a subset of
$\R$ associated to curves $\gamma_h$ with
unbounded component $v$ near the singular line $\{u=0\}\times \R$.
They provide elementary forms to construct TWS with
cusps.}\label{Fig_sing-III}
\end{figure}

\section{Singular traveling wave solutions}\label{S_STW}

This section is devoted to the characterization of singular TWS of equation
\eqref{E_TWeq}.  We are interested in singular
TWS such that either $v_h^\pm(0)$ is well-defined  for $h = h_0$, or
 $\lim_{u\rightarrow 0^+} v_h^\pm(u)=\pm \infty$
 for $h\neq h_0$.

\subsection{Peaked waves}\label{S_peaked}

A  TWS  of \eqref{E_TWeq} given by a function $u:\R\rightarrow\R$ is
called \emph{peaked}  if it is smooth except at a finite or
countable number of points (\emph{peaks}) $\mathcal{S}=\{t_k\in\R,\,
k\in\Z \}$ where
$$
    0 \neq \lim_{t\rightarrow t_k^+} u'(t) = -  \lim_{t\rightarrow t_k^-} u'(t) \neq \pm \infty,
$$
The main result of this section is the following characterization:
\begin{Proposition}\label{P_peaked} The equation \eqref{E_TWeq}
has peaked TWS if and only if
\begin{itemize}
\item  $F'(0)<0$ and
\item   there exists $m>0$ such that $F(m)=F(0)$ and $F(u)<F(0)$ for $u\in
(0,m)$.
\end{itemize}
These solutions are either
 \begin{enumerate}[(i)]
 \item peaked periodic, with period
 $T=\displaystyle{\frac{2}{\sqrt{2}}\int_0^m\sqrt{\frac{u}{F(0)-F(u)}}du}$, if and only if in addition $F'(m)\neq
  0$, or
  \item peaked solitary if and only if in addition
  $F'(m)=0$.
 \end{enumerate}
These solutions are weak singular TWS and they are analytic except
for a discontinuity in the first derivative at the peaks.
Furthermore, the peaked solitary waves are symmetric with respect to
their unique maximum and decay exponentially to zero at infinity.
Peaked periodic solutions have a unique maximum and minimum per
period and are symmetric with respect to these local extrema.
\end{Proposition}

\begin{proof} Peaked TWS are compositions of the
elementary forms studied in the case (a) of Section
\ref{S_ElemForms}. They appear as a consequence of the existence of values  $t\in\R$ such that the solutions of \eqref{E_TWeq} reach $u=0$  with non-vanishing one-sided derivative.
These solutions are associated to the integral curves of
\eqref{E_TWSys} with energy $h=h_0$. The branches $(u,v_{h_0}^\pm(u))$
must be defined and bounded, intersecting the singular line $\{u=0\}$ at finite points $(0,+a)$ and $(0,-a)$ different from $(0,0)$,
so that the corresponding solutions reach $0$ with
non-vanishing derivative.  Therefore, there must exist $m>0$ such that $F(u)<F(0)$ for
$u\in(0,m)$ and $F(m)=F(0)$, and additionally $F'(0)<0$.  These two
necessary conditions are also sufficient for the existence of peaks,
as we will see in the construction of solutions in the proofs of (i)
and (ii).

Suppose that $F'(m)\neq 0$, so  $(m,0)$ is a regular point and
therefore, by the symmetry of system \eqref{E_TWSys} and by equation
\eqref{E_fin-time}, there is an orbit of system \eqref{E_TWSys}
connecting $(0,+a)$ with $(0,-a)$ in finite time. This means that
there exist $t_1,t_2 < \infty$ such that the first component of the
solution of system \eqref{E_TWSys}, $u(t)$, is a strong solution of
equation \eqref{E_TWeq} (and equation \eqref{E_singTW} with $h=h_0$)
with maximal interval of definition given by $t\in(t_1,t_2)$. This
solution is the elementary form studied in the case (a1) in Section
\ref{S_ElemForms}, hence $\lim_{t\rightarrow t_1^+}u(t) = 0$ and
$\lim_{t\rightarrow t_1^+}\dot{u}(t) = a$. By symmetry, $
\lim_{t\rightarrow t_2^-}u(t) = 0$, $\lim_{t\rightarrow
t_2^-}\dot{u}(t) =-a$ and $u\left(\frac{t_1+t_2}{2}\right)=m$.
 Lemma \ref{L_uniq2} allows us to extend the above solution continuously to
$\R$ by gluing together copies of it defined in the intervals
$(t_k,t_{k+1})$ with $k\in\Z$, where $t_k=t_1+k(t_2-t_1)$. This
leads to the function
\begin{equation}\label{E_Periodicpeaked}
\tilde{u}(t)=\left\{
       \begin{array}{ll}
         u\left(t-(k-1)(t_2-t_1)\right) & \mbox{for } t\in (t_k,t_{k+1}), \\
         0 & \mbox{for } t=t_k,
       \end{array}
     \right.
\end{equation}
which is periodic with period
$$
T=t_2-t_1=\frac{2}{\sqrt{2}}\int_0^m\sqrt{\frac{u}{F(0)-F(u)}}\, du.
$$
By Lemma \ref{L_finite_time} (ii) the above integral is well
defined, since $F'(m)\neq 0$. Notice also that $\lim_{t\rightarrow
t_k}\tilde{u}(t) = 0$ and $\lim_{t\rightarrow
t_{k}^\pm}\dot{\tilde{u}}(t) =\pm a$, so $\tilde{u}(t)$ is
peaked periodic. Furthermore, the solution attains a unique
maximum $m$ at $t=(t_k+t_{k+1})/2$ and is symmetric with respect to
it on  each period. Observe that $\tilde{u}\in H^1_{loc}(\R)$.
Indeed, the function $\tilde{u}$ is bounded and analytic in
$\R\setminus \mathcal{S}$ where $\mathcal{S}=\{t_k,\, k\in\Z\}$, and
since $\lim_{t\rightarrow t_k} (\dot{\tilde{u}}(t))^2=a^2$, also
$|\dot{\tilde{u}}|^2$ is bounded in $\R\setminus\mathcal{S}$.
Recalling in addition that $\tilde{u}(t)$ is a strong solution of
\eqref{E_TWeq} on each interval $(t_k,t_{k+1})$, cf.~(a1) in Section
\ref{S_ElemForms}, we get that  it is a TWS in the sense of
Definition \ref{D_TWS}. By construction $\tilde{u}(t)$ satisfies
equation \eqref{E_singTW} with $h=h_0$ except at the points in
$\mathcal{S}$ where the equation holds in the limit \eqref{E_weaksingular}. So
according to Definition \ref{D_singularTWS},  it is a weak singular
TWS of  \eqref{E_TWeq}. Finally, observe that the solution
$\tilde{u}$ constructed above is the only possible continuous
continuation of $u(t)$ on $\mathbb{R}$.

Suppose now that $F'(m)= 0$, so $(m,0)$ is a critical point of
\eqref{E_TWSys}. In this case there exist two orbits connecting
$(m,0)$ with $(0,+a)$ and $(0,-a)$, respectively, cf.~the elementary
forms of (a2) in Section \ref{S_ElemForms}. Again, Lemma
\ref{L_finite_time} (i) guarantees that any point on these orbits
reaches $(0,\pm a)$ in finite time, but it takes an infinite time to
reach $(m,0)$. So there exist
strong solutions $u_{\pm}(t)$ of \eqref{E_TWeq} defined on
$(t_1,\infty)$ and $(-\infty,t_2)$, respectively. The non-uniqueness of solutions then allows
us to choose $t_2=t_1$ to  construct the
function
$$
u(t)=\left\{
       \begin{array}{ll}
         u_-(t) & \mbox{for } t\in (-\infty,t_1], \\
         u_+(t) & \mbox{for } t\in (t_1,+\infty),
       \end{array}
     \right.
$$
which is a peaked solitary wave defined on $\R$. The same
arguments as before show that $u$ is in  $H^1_{loc}(\R)$, it is a
weak solution of \eqref{E_TWeq}, and in particular  a weak singular
TWS of
 \eqref{E_TWeq}. Again, $u$  is the only
possible continuous continuation of $u_-$ and $u_+$ onto $\R$ .
~\end{proof}

\subsection{Solitary waves with compact support  and associated composite waves}\label{S_compactsupport}

A  TWS  of  \eqref{E_TWeq} given by a function $u:\R\rightarrow\R$
has \emph{compact support} if  there exist $-\infty < t_1 < t_2 <
\infty$ such that $u$ is constant on $\R\setminus[t_1,t_2]$ and $u$
is non constant on $[t_1,t_2]$. We call $[t_1,t_2]$ the support of $u$. A
solitary wave with compact support, or simply \emph{compact solitary
wave}, is a continuous TWS which has compact support and a unique
extremum. We say that a TWS $\tilde{u}(t)$ is a composite wave
associated to a compact solitary wave $u(t)$ with support
$[t_1,t_2]$, or simply a \emph{composite wave of} $u(t)$, if it is
obtained by gluing together copies of one compact solitary wave in
such a way that the supports of each copy do not overlap. More
precisely, $\tilde{u}(t)$ is a composite wave of a solitary wave
$u(t)$ if there exists a collection of intervals
$I_k=\left(t_1-a_k(t_2-t_1),t_2-a_k(t_2-t_1)\right)$ for $a_k\in\R$
and $k\in K$, where $K$ is either $\Z$ or a finite collection of
indices, with $I_k\cap I_j\neq \emptyset$ for all $k\neq j$, such
that
$$
\tilde{u}(t)=\left\{
       \begin{array}{ll}
         u(t+a_k(t_2-t_1)) & \mbox{if } t\in I_k\mbox{ for some }k\in K, \\
         0 & \mbox{if } t\in \R\setminus \cup_{k\in
         K}I_k.
       \end{array}
     \right.
$$
The following result characterizes the compact solitary waves and composite waves.

\begin{Proposition}\label{P_cpsupp}
The equation  \eqref{E_TWeq} has compact solitary TWS and their
associated composite solutions if and only if
\begin{itemize}
\item  $F'(0)=0$  and  $F''(0)<0$, and
\item   there exists $m>0$ such that $F(m)=F(0)$, $F(u)<0$ for $u\in
(0,m)$ and $F'(m)\neq 0$.
\end{itemize}
These solutions are either compact solitary waves, composite
multi-bump solutions with compact support or composite waves with
non-compact support that can be either aperiodic or periodic with
arbitrary period. These solutions are
strong singular TWS, they are  $\C^1(\R)$ and piecewise analytic.
Furthermore, the compact solitary waves are symmetric with respect
to their unique maximum.
\end{Proposition}

\begin{proof}
Compact  solitary TWS appear when there exist values  $t\in\R$
such that the solutions of \eqref{E_TWeq} reach $u=0$
with vanishing one-sided derivative. These solutions are associated
to the integral curves of \eqref{E_TWSys} with energy $h=h_0$,
and as before, the branches $(u,v_{h_0}^\pm(u))$ must be defined and
bounded such that they intersect the singular line $\{u=0\}\times
\R$ at $(0,0)$ in finite time. This last fact guarantees that the
corresponding solution of \eqref{E_TWeq} reaches $u=0$ with
vanishing derivative in finite time.

As studied in the case (b) of Section \ref{S_ElemForms},  the existence and boundedness  of such curves is guaranteed if there
exists $m>0$ such that $F(u)<F(0)$ for $u\in(0,m)$ and $F(m)=F(0)$.
The intersection of the curves $(u,v_{h_0}^\pm(u))$ with the
singular line $\{u=0\}\times \R$ at $(0,0)$ is ensured if $F'(0)=0$, since in
this case
\begin{equation}\label{E_tozero}
\lim\limits_{u\rightarrow 0^+} v_{h_0}^\pm(u)=
\lim\limits_{u\rightarrow 0^+}
\pm\sqrt{\frac{-2\,F''(u)u^2+o(u^3)}{u}}=0.
\end{equation}
By Lemma \ref{L_finite_time} (i) and (ii), the orbits of
\eqref{E_TWSys} connect $(0,0)$ with $(m,0)$  in finite time if
$F''(0)<0$ and $F'(m)>0$. Under these conditions there exist $t_1,
t_2\in\R$ such that the first component of the solution of system
\eqref{E_TWSys}, $u(t)$, is a strong solution of \eqref{E_TWeq} (and
also \eqref{E_singTW} with $h=h_0$) in the maximal interval of
definition $t\in(t_1,t_2)$. Moreover, $\lim_{t\rightarrow t_1^+}u(t)
= 0$ and by symmetry $
     \lim_{t\rightarrow t_2^-}u(t) = 0$ and
 $u\left(\frac{t_1+t_2}{2}\right)=m$ is the unique maximum. This solution $u(t)$ is
 the elementary form that appears in the case (b1) of Section
 \ref{S_ElemForms}.

In particular, the above conditions imply that equation
\eqref{E_TWeq} is not unique at $u=0$, and that
$u(t)\equiv 0$ is also a solution. Therefore it is possible to obtain a continuous continuation of $u(t)$ with compact support in $\R$. We define the
function
$$
\tilde{u}(t)=\left\{
       \begin{array}{ll}
         u(t) & \mbox{for } t\in (t_1,t_2), \\
         0 & \mbox{for } t\in \R\setminus(t_1,t_2).
       \end{array}
     \right.
$$
 Observe that
$\tilde{u}(t)\in H^1_{loc}(\R)$ and it is analytic in $\R\setminus
\{t_1,t_2\}$. Furthermore,  $\lim_{t\rightarrow t_1^+}\dot{u}(t) =
0$ and  $\lim_{t\rightarrow t_2^-}\dot{u}(t) = 0$ in view of
\eqref{E_tozero}.  Hence $\tilde{u}(t)\in\mathcal{C}^1(\R)$, and by
construction it is a strong singular TWS of \eqref{E_TWeq}. In a
similar way as above, we can  also construct the following composite
waves: multi-bump waves with  compact support by gluing together a
finite number of copies of $u(t)$, and waves with non-compact
support that can be either aperiodic or periodic with arbitrary
period.
All these solutions are strong singular TWS of \eqref{E_TWeq}, they
are piecewise analytic and at most $\mathcal{C}^1(\R)$,~cf. Remark
\ref{R_cpsupp} below.

The sufficiency part of the proof follows from the fact that if $F$
satisfies the stated conditions, then it is possible to find the
desired solutions following the above construction.~\end{proof}

As a direct consequence of Propositions \ref{P_peaked} and
\ref{P_cpsupp}, we obtain that it is not possible to find a peaked
wave with compact support.

\begin{Corollary} \label{C_coex}
The equation  \eqref{E_TWeq} cannot admit peaked and compactly
supported  TWS for the same $F$. In particular, a peaked TWS can not
have compact support, and conversely,
 a TWS with compact support can not have peaks.
\end{Corollary}

\begin{remark}\label{R_cpsupp}
In view of Lemma \ref{L_finite_time}, the conditions of Proposition
\ref{P_cpsupp} guarantee the existence of a homoclinic orbit of
system \eqref{E_TWSys} where every point on the orbit is connected to $(0,0)$ in a finite time. They
also account for the fact that the compactly supported solutions are
at most $\C^1$. Indeed, notice that under the conditions of Proposition \ref{P_cpsupp} we have that

\begin{equation*}
 \dot u(t) = v_{h_0}^\pm= \pm \sqrt{-2F''(u) u+o(u^2)}
\end{equation*}
and
\begin{align*}
 \ddot u(t)
         &= \pm \frac{-2 F^{(3)}(u)u\dot u-2F''(u)\dot u + o(u)\dot u}{2(\pm \dot u)}\\
         &= -F^{(3)}(u)u- F''(u)+o(u).
\end{align*}
Therefore,
\begin{align}
\label{E_limusegunda}
 \lim_{t\rightarrow t_1^-}\ddot u(t) &= -F''(u) \neq 0 =  \lim_{t\rightarrow t_1^+}\ddot u(t) \notag\\
 \lim_{t\rightarrow t_2^+}\ddot u(t) &= -F''(u) \neq 0 =  \lim_{t\rightarrow t_1^-}\ddot u(t) ,
\end{align}
and hence the continuation to $u(t) \equiv 0$ is not $\C^2$ since $F''(u) <0$.
If we were to demand a $\C^2$-continuation, this would
necessarily require the second derivative of $F$ to vanish.
 As a consequence, the solutions would
loose the compactness property, since in view of Lemma
\ref{L_finite_time} the existence time of the loop would become
infinite.
\end{remark}

\begin{remark}
Observe that a different kind of composition of locally defined solutions may also be obtained under conditions stated in Proposition \ref{P_cpsupp} but with $F'(0)<0$. In this case
 $$ \lim\limits_{u\rightarrow 0^+} v_{h_0}^\pm(u)=
\lim\limits_{u\rightarrow 0^+}
\pm\sqrt{\frac{-2\,F'(u)u+o(u^2)}{u}}=\pm\sqrt{-2F'(0)}=:\pm a\neq
0,
$$
and hence $\lim_{t\rightarrow t_1^+}\dot{u}(t) = a$ and
$\lim_{t\rightarrow t_2^-}\dot{u}(t) = -a$. However, $u(t)\equiv 0$
is not a solution of \eqref{E_TWeq} under this assumption on $F$ and
therefore the only possible continuous continuation is given by the
periodic function \eqref{E_Periodicpeaked}. But this is not a composition of a
 compact solitary wave, i.e~not a composite wave as defined at the beginning of this section.
\end{remark}

\subsection{Fronts with finite-time decay and plateau-shaped waves}\label{S_NovesC1}

\begin{Proposition}\label{P_c1}
The equation  \eqref{E_TWeq} has  fronts, solitary and
plateau-shaped singular TWS solutions  if and only if
\begin{itemize}
\item  $F'(0)=0$  and  $F''(0)<0$, and
\item   there exists $m>0$ such that $F(m)=F(0)$, $F(u)<0$ for $u\in
(0,m)$ and $F'(m)= 0$.
\end{itemize}
The solitary waves are strong solutions of \eqref{E_TWeq} and
therefore analytic. The fronts and plateau-shaped waves are strong
singular TWS which are at most $\C^1$ and piecewise analytic.
Moreover, the fronts have finite-time decay on one side.
\end{Proposition}

\begin{proof}
Under the present hypothesis, there exist orbits of system
\eqref{E_TWSys} corresponding to the branches $(u,v^{\pm}_{h_0}(u))$,
which connect $(0,0)$ with the critical point $(m,0)$. By  Lemma
\ref{L_finite_time}, any point on these orbits is connected to $(0,0)$
in finite time and to $(m,0)$ in infinite time. They
yield two strong solutions $u_{\pm}(t)$ with maximal interval of
definition $(t_1,\infty)$ and $(-\infty,t_2)$, respectively, which
are given by the elementary forms of the case (b2) of Section
\ref{S_ElemForms}. By gluing them together with the solution
$u(t)\equiv 0$ we get the following solutions: fronts given by
$$
 \tilde{u}_-(t)=\left\{
       \begin{array}{ll}
       u_-(t) & \mbox{for } t\in (-\infty,t_1),\\
         0 & \mbox{for } t\in [t_1,\infty),
       \end{array}
     \right.  \mbox{  and  }\,
 \tilde{u}_+(t)=\left\{
       \begin{array}{ll}
         0 & \mbox{for } t\in (-\infty,t_2],\\
         u_+(t) & \mbox{for } t\in (t_2,\infty).
       \end{array}
     \right.
$$
Setting $t_1=t_2$ we get a solitary wave  given by
$$
u(t)=\left\{
       \begin{array}{ll}
         u_-(t) & \mbox{for } t\in (-\infty,t_1),\\
         0 & \mbox{for } t=t_1,\\
         u_+(t) & \mbox{for } t\in (t_1,\infty),
       \end{array}
     \right.
$$
which is symmetric with respect to its unique maximum.
Choosing $t_1<t_2$ we get a plateau-shaped solitary wave given by
$$
u(t)=\left\{
       \begin{array}{ll}
         u_-(t) & \mbox{for } t\in (-\infty,t_2),\\
         0 & \mbox{for } t\in [t_2,t_1],\\
         u_+(t) & \mbox{for } t\in (t_1,\infty),
       \end{array}
     \right.
$$
The front and plateau-shaped solutions are analytic in
$\R\setminus\{t_i\}$ and satisfy equation \eqref{E_singTW} for
$h=h_0$, hence they are strong singular TWS of \eqref{E_TWeq}.
Furthermore, they are at most  $\mathcal{C}^1(\R)$ for the same
reason explained in Remark \ref{R_cpsupp}. The solitary wave
solutions, however, are in fact strong solutions of \eqref{E_TWeq}
on $\R$ since $\ddot u(t)$ is defined everywhere,
cf.~\eqref{E_limusegunda}, and therefore they are analytic on $\R$.
~\end{proof}

We emphasize that the fronts described above are not the classical smooth fronts with exponential decay on both ends, but they decay in finite time on one end.

%

\subsection{Cusped waves}\label{S_cusped}

A TWS  of  \eqref{E_TWeq} given by a function $u:\R\rightarrow\R$ is called
\emph{cusped}  if it is smooth except at a finite or countable
number of points (\emph{cusps}) $\mathcal{S}=\{t_k\in\R,\, k\in\Z
\}$ where
  \begin{equation*}
        \lim_{t\rightarrow t_k^+} u'(t) = -  \lim_{t\rightarrow t_k^-} u'(t) = \pm
        \infty.
  \end{equation*}

\begin{Proposition}\label{P_cusped} The equation  \eqref{E_TWeq}
has cusped TWS if and only if
 there exists $m>0$ such that $F(m)-F(u)>0$ for all $u\in[0,m)$. These solutions are either
 \begin{enumerate}[(i)]
 \item cusped periodic with period $T=\displaystyle{\frac{2}{\sqrt{2}}\int_0^m\sqrt{\frac{u}{F(m)-F(u)}}du}$ if and only if in addition $F'(m)\neq  0$, or
  \item cusped solitary if and only if in addition
  $F'(m)=0$.
 \end{enumerate}
These solutions are weak singular TWS and they are analytic except
for a discontinuity in the first derivative at the peaks.
Furthermore, the cusped solitary waves are symmetric with respect to
their unique maximum and decay exponentially to zero at infinity.
Cusped periodic solutions have a unique maximum and minimum per
period and are symmetric with respect to these local extrema.

\end{Proposition}

\begin{proof}

Cusped TWS correspond to orbits of \eqref{E_TWSys} which are defined
by the curves  $(u,v_h^{\pm}(u))$  with $h\neq h_0$ satisfying
$$
\lim\limits_{u\rightarrow 0^+} v_h^\pm(u)= \lim\limits_{u\rightarrow
0^+} \pm\sqrt{2\,\frac{h-F(u)}{u}}=\pm \infty.
$$
Therefore, a necessary condition for the appearance of cusps is that
there exists $m>0$ such that $F(m)-F(u)>0$ for $u\in[0,m)$. In
particular, this implies that $h_m:=F(m)>F(0)$. As shown in the case
(c) of Section \ref{S_ElemForms} the  orbits
 of \eqref{E_TWSys} passing through the points
 $(\bar{u},v_{h_m}^\pm(\bar{u}))$ are unbounded
in the component $v$ but approach the singular line $u=0$ at infinity in
a finite time.  We will see in the construction of solutions in the
proofs of (i) and (ii) below that the necessary condition deduced
above is also sufficient for the existence of cusped TWS.

(i) If $F'(m)\neq 0$ then the point $(m,0)$ connecting the two
branches $(u,v_{h_m}^\pm(u))$ is regular, so by equation
\eqref{E_fin-time-2} there exist $t_1,t_2 < \infty$ such that the
first component of the solution of
     system  \eqref{E_TWSys}  $u(t)$, is a strong solution of \eqref{E_TWeq} (and also of
     \eqref{E_singTW} with $h=h_m$), with maximal interval of definition
      $t\in(t_1,t_2)$, which is given by the elementary form that appears in case (c1)
     of Section \ref{S_ElemForms}.
Notice that  $u(t)\equiv 0$ is not a solution of
     \eqref{E_singTW} for $h=h_m>0$.
     Therefore, the only possible continuous
     continuation of $u(t)$ preserving the energy is given by
     the periodic function
$$
\tilde{u}(t)=\left\{
       \begin{array}{ll}
         u\left(t-(k-1)(t_2-t_1)\right) & \mbox{for } t\in (t_k,t_{k+1}), \\
         0 & \mbox{for } t=t_k,
       \end{array}
     \right.
$$
where $t_k=t_1+k(t_2-t_1)$ for $k\in \Z$, which is periodic with period
$$T=t_2-t_1=\displaystyle{\frac{2}{\sqrt{2}}\int_0^m\sqrt{\frac{u}{F(m)-F(u)}}du}.$$
Observe now that integrating system \eqref{E_TWSys}, and denoting
$u_\varepsilon:=u(t_1+\varepsilon)$ we get,
$$
\int_{t_1}^{t_1+\varepsilon} \dot{u}^2dt=\int_{0}^{u_\varepsilon}
v^2 \frac{du}{v}=\int_{0}^{u_\varepsilon}
\sqrt{\frac{2(F(m)-F(u))}{u}}du.$$ Since $F(m)-F(u)>0$ for all
$u\in[0,u_\varepsilon]$ and recalling the criterion used in Lemma \ref{L_finite_time}, we find that  $\lim_{u\rightarrow 0^+}
\sqrt{{(F(m)-F(u))}/{u}}\cdot \sqrt{u}=A$ with $A\neq 0$ and $A\neq
\infty$. Therefore,
$$
\int_{0}^{u_\varepsilon} \sqrt{\frac{2(F(m)-F(u))}{u}}du <\infty,
$$ and hence $\int_{t_1}^{t_1+\varepsilon} \dot{u}^2dt<\infty$.
Analogously $\int_{t_2-\varepsilon}^{t_2} \dot{u}^2dt<\infty$ and
hence $\int_{t_k-\varepsilon}^{t_k+\varepsilon}
(\dot{\tilde{u}})^2dt$ is convergent for every $t_k$ so that for
each compact $K\subset \R$ we have
$$
\int_K (\dot{\tilde{u}})^2dt<\infty.$$ Since $\tilde{u}(t)$ is
bounded as well, this solution is  in $H^1_{loc}(\R)$, and  analytic on
$\R\setminus \mathcal{S}$ where
$\mathcal{S}=\{t_k\in\R,\, k\in\Z \}$. Finally observe that
$\tilde{u}(t)$  is a  weak singular TWS of  \eqref{E_TWeq},
because
$$
\lim\limits_{t\rightarrow t_k} \frac{u (\dot{u})^2}{2}+F(u)= u\,
\frac{F(m)-F(u)}{u}+F(u)=F(m)=h_m.
$$

 (ii) If $F'(m)=0$ then the orbits corresponding to
$(u,v_{h_m}^\pm)$  approach the singular line $\{u=0\}$ in finite time, but
it takes an infinite time to reach $(m,0)$. So there exist two  strong solutions of \eqref{E_TWeq}  $u_\pm(t)$ given by the elementary forms of the case
(c2) of Section \ref{S_ElemForms}, defined on the maximal intervals
$(t_1,+\infty)$ and $(-\infty,t_2)$, respectively. The only way to
construct a continuous continuation in $\R$ preserving the energy is
by choosing $t_2=t_1$ and gluing together the corresponding
solutions $u_{\pm}$. These considerations lead us to define the
function
$$
\tilde {u}(t)=\left\{
       \begin{array}{ll}
         u_-(t) & \mbox{for } t\in (-\infty,t_1), \\
         0& \mbox{for } t=t_1\\
         u_+(t) & \mbox{for } t\in (t_1,+\infty),
       \end{array}
     \right.
$$
which is a cusped  solitary TWS defined in $\R$. The same
arguments as in the proof of statement (a) show that $\tilde{u}$ is
a weak singular TWS of  \eqref{E_TWeq}. ~\end{proof}

\subsection{Exhaustivity of the characterization}\label{S_exhaustivity}

Observe that the elementary forms presented in Section
\ref{S_ElemForms} capture all the strong solutions of equation
\eqref{E_TWeq} reaching or tending to $\{u=0\}$, whose maximal
interval of definition is not $\R$. In addition, the singular
solutions described in the preceding Section
\ref{S_STW} cover all possible continuous
extensions to $\R$ on the same energy level, using these elementary forms and the constant function $u(t)\equiv 0$
whenever it is a solution. In
consequence, our characterization of singular TWS
for equation \eqref{E_TWeq} given in Propositions
\ref{P_peaked}, \ref{P_cpsupp}, \ref{P_c1}  and  \ref{P_cusped} is exhaustive.

\section{Application to shallow water equations}\label{S_bif_approach}
The aim of this section is to demonstrate the applicability of the
propositions developed in the preceding sections. We exemplify
our approach by studying the equation
for surface waves of moderate amplitude in shallow water and the Camassa-Holm equation. In particular, we
show how the different types of singular TWS can be obtained varying
the energy of the corresponding Hamiltonian systems. This approach may be applied to study singular TWS of a variety of other equations, for example a class of
nonlinear wave equations related to the inviscid Burgers' equation
and Camassa-Holm equation studied in \cite{Lenells2006}, the family of equations analyzed in \cite{LiOlv}, and a generalization of the Camassa-Holm equation studied in \cite{Liu2004}. In the latter paper, the authors conclude with a conjecture on the non-existence of peaked solitary solutions when a certain parameter becomes non-positive. In view of the results in Section \ref{S_STW} we are able to give an affirmative answer.

\subsection{Surface waves of moderate amplitude in shallow water}
\label{S_MASE}
In this section we study singular TWS of the equation for surface waves of moderate amplitude in shallow water,
\begin{align}
  \label{E_MASE}
  u_t + u_x + 6u u_x - 6u^2u_x + 12u^3u_x+
      u_{xxx} -u_{xxt}  + 14uu_{xxx} + 28u_xu_{xx} = 0,
\end{align}
which was first derived by Johnson \cite{Joh02}, whose considerations were extended by Constantin and Lannes \cite{ConLan09}. We refer to \cite{Gey12c} for a first study of smooth solitary waves and to \cite{Gasull2014} for a more extensive characterization of TWS of equation \eqref{E_MASE}.
We  introduce the traveling wave Ansatz  $u(x,t)=u(x-c\,t)$ and integrate  once to obtain
  \begin{equation}
  \label{E_MASEODE}
         u''\Big(u+\frac{1+c}{14}\Big) +\frac{1}{2}(u')^2
      +  K + (1-c) u +3u^2-2u^3 +3u^4 = 0,
\end{equation}
for some constant $K\in\R$. Notice that after the change of
variables $u \mapsto u -\frac{1+c}{14}$ the above equation is of the
form \eqref{E_TWeq}, with $F$ a suitable polynomial in $u$ depending on the  parameters $c$ and $K$.
The relation between the parameters and the
qualitative properties of $F$  is studied in
detail in \cite{Gasull2014}. In particular, it is observed that $F$ has either no extremum or there are two extrema which we denote by $p_1$ (the local maximum) and $p_2$ (the local minimum of $F$) such that $p_1<p_2$. Let  $h_i=F(p_i)$ for $i=1,2$.
We distinguish different cases
depending on the position of $p_1$ and $p_2$ with respect to $u=0$, and the sign of $h_0-h_2$.
Taking into account these cases and the characterization given in Propositions \ref{P_peaked}, \ref{P_cpsupp} and  \ref{P_cusped}  we obtain the following types of bounded singular TWS for equation \eqref{E_MASE} varying the energy. We point out that, as a consequence of Corollary \ref{C_coex},  compact solitary waves and peaked waves cannot coexist within a single case.   To give an example, Figure \ref{Fig_FMASE1} shows the different TWS of equation \eqref{E_MASE} that appear for different energy levels when $ 0<p_1<p_2$ and $h_2>h_0$.

  \vspace{1em}

 \hspace{-1.7em}
\begin{minipage}{0.5 \textwidth}
  \begin{tabular}{| l || c | }
    \hline
Energy/Case                & $ 0<p_1<p_2$ and $h_2>h_0$                                \\\hline\hline
$h>h_1$      & cusped periodic          \\ \hline
$h=h_1$      & cusped \& smooth solitary      \\\hline
$h_1>h>h_2$ & cusped \& smooth periodic \\ \hline
$h=h_2$         & cusped periodic \& constant    \\ \hline
$h_2>h>h_0$ & cusped periodic  \\ \hline
$h\leq h_0$    &                  \\ \hline
  \end{tabular}
\end{minipage}
\begin{minipage}{0.5\textwidth}
   \begin{tabular}{| l || c |}
    \hline
           Energy/Case                        &  $ 0<p_1<p_2$ and $h_2<h_0$  \\ \hline\hline
     $h>h_1$              & cusped periodic  \\ \hline
     $h=h_1$              & cusped \& smooth solitary \\  \hline
     $h_1>h>h_0$  &  cusped \& smooth periodic \\  \hline
     $h_0\geq h> h_2$             & smooth periodic \\ \hline
     $h=h_2$         & constant \\   \hline
     $h \leq h_2$  &       \\   \hline
  \end{tabular}
  \end{minipage}

  \vspace{1em}
  \hspace{-2.2em}
\begin{minipage}{0.535\textwidth}
 \begin{tabular}{| l || c |}
    \hline
           Energy/Case                   &  $ 0<p_1<p_2$ and $h_2=h_0$   \\ \hline\hline
    $h>h_1$  & cusped periodic  \\ \hline
     $h=h_1$       & cusped \& smooth solitary   \\  \hline
        $h_1>h>h_0=h_2$ & cusped \& smooth periodic \\\hline
     $h = h_0$         & constant     \\ \hline
     $h < h_0$  &        \\   \hline
  \end{tabular}
  \end{minipage}
  \begin{minipage}{0.5\textwidth}
   \begin{tabular}{| l || c |}
    \hline
           Energy/Case                        &   $ 0=p_1<p_2$ \\ \hline\hline
     $h>h_1$              & cusped periodic  \\ \hline
     $h=h_1=h_0$              &  compact  solitary \& composite \\  \hline
     $h_1>h>h_2$  & smooth periodic \\  \hline
     $h=h_2$         & constant \\   \hline
     $h < h_2$  &      \\   \hline
  \end{tabular}
  \end{minipage}

  \vspace{1em}

  \begin{minipage}{0.5\textwidth}
 \begin{tabular}{| l || c | c |  c | }
    \hline
           Energy/Case                        &   $ p_1<0<p_2$   &   $ p_1<p_2=0$   &   $ p_1<p_2<0$   \\ \hline\hline
     $h>h_0$              & cusped periodic    & cusped periodic    & cusped periodic  \\ \hline
     $h=h_0$              & peaked periodic & constant & \\  \hline
     $h_0>h>h_2$     &  smooth periodic & & \\  \hline
     $h=h_2$             & constant & &\\ \hline
     $h < h_2$  & &  &   \\   \hline
  \end{tabular}
  \end{minipage}

  \begin{figure}[ht!]
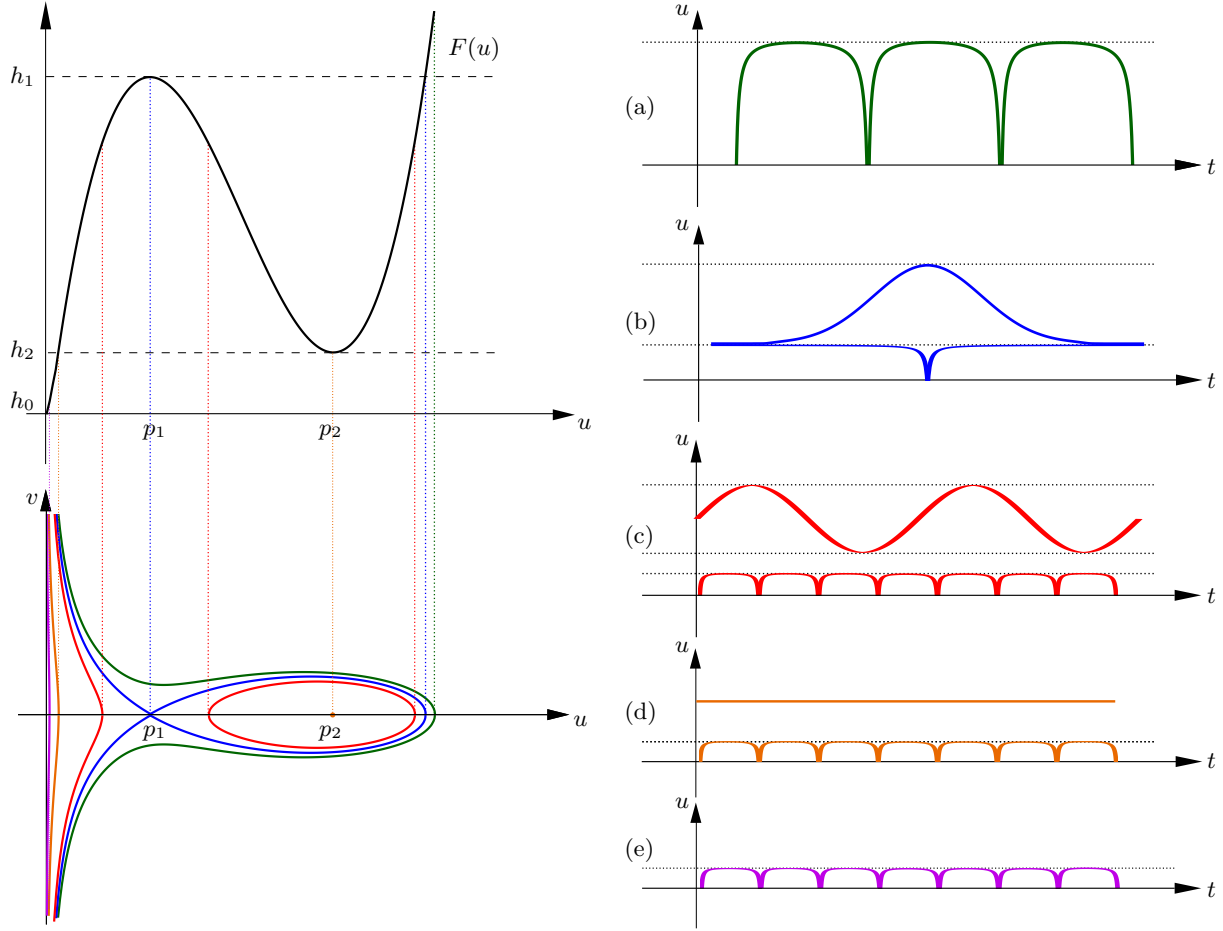

\footnotesize
  \centering
    \begin{lpic}[l(0mm),r(0mm),t(-5mm),b(0mm)]{FigureMASE(0.52)}
    \lbl[l]{-2,135;$h_0$}
    \lbl[l]{-2,148;$h_2$}
    \lbl[l]{-2,218;$h_1$}

    \lbl[l]{32,127;$p_1$}
    \lbl[l]{77,127;$p_2$}

    \lbl[l]{143,129;$u$}
    \lbl[l]{110,225;$F(u)$}

    \lbl[l]{32,50;$p_1$}
    \lbl[l]{77,50;$p_2$}

    \lbl[l]{142,53;$u$}
    \lbl[l]{2,110;$v$}

    \lbl[l]{168,235;$u$}
    \lbl[l]{304,195;$t$}
    \lbl[l]{155,210;(a)}

    \lbl[l]{168,180;$u$}
    \lbl[l]{304,140;$t$}
    \lbl[l]{155,155;(b)}

    \lbl[l]{168,125;$u$}
    \lbl[l]{304,85;$t$}
    \lbl[l]{155,100;(c)}

    \lbl[l]{168,72;$u$}
    \lbl[l]{304,43;$t$}
    \lbl[l]{155,55;(d)}

    \lbl[l]{168,32;$u$}
    \lbl[l]{304,10;$t$}
    \lbl[l]{155,20;(e)}
    \end{lpic}
\caption{Singular TWS of equation \eqref{E_MASE} varying the energy level $h$ in the case $ 0<p_1<p_2$ and $h_2>h_0$. (a)  cusped periodic waves for $h>h_1$; (b) cusped and smooth solitary waves  for $h=h_1$; (c) cusped and smooth periodic waves for $h_1>h>h_2$; (d)  cusped periodic waves and constant solutions for  $h=h_2$; (e) cusped periodic waves for $h_2>h>h_0$.
}
\label{Fig_FMASE1}
\end{figure}

\subsection{The Camassa-Holm Equation}
\label{S_CH}
In the present section we will study singular TWS of  the Camassa-Holm equation (CH)
\begin{equation}
 \label{CH}
  u_t + 2\kappa\,u_x -u_{txx} + 3\, u\,u_x  = 2\,u_x u_{xx} + u\,u_{xxx},
\end{equation}
for $x \in \R$, $t>0$ and $\kappa\in\R$, which was introduced in the context of water waves by Camassa and Holm \cite{Camassa1993}. For a classification of weak traveling wave solutions of the Camassa--Holm equation we refer to \cite{Len04}. Proceeding as in the previous section we introduce the traveling wave Ansatz  $u(x,t)=u(x-c\,t)$. Integrating once equation \eqref{CH} takes the form
\begin{equation*}
 u''(u-c) + \frac{(u')^2}{2} + r + (c-2\kappa)\,u - \frac{3}{2}u^2 = 0,
\end{equation*}
where $r$ is a constant of integration. The change of variables
\[
  w=u-c,
\]
transforms the above equation to the form \eqref{E_TWeq}  with
\begin{equation}
 \label{FCH}
  F(w)=  A\,w + B\,w^2 - \frac{1}{2}w^3,
\end{equation}
where $A=r-2\kappa c - \frac{1}{2}c^2$ and $B=-(c+\kappa)$. $F(w)$ is a third order polynomial which satisfies $F(0)=0$, it has at most three roots
\begin{equation*}
 w=0 \text{ and } w= B \pm \sqrt{B^2+2A},
\end{equation*}
and at most two extrema
\begin{equation*}
 p_i= \frac{2B +(-1)^ i\sqrt{4B^2+6A}}{3},\; i=1,2.
\end{equation*}
We may assume that $B\geq 0$ since otherwise the change of variables
$(\hat{w},\hat{v}) = -(w,v)$ yields this situation (this is
equivalent to considering system \eqref{E_TWSys} only for $u\geq
0$).
Note that $F$ does not have any extremum when  $4B^2 +6A\leq 0$, and it has two distinct extrema otherwise. In the latter case, $p_1$ is the local minimum and $p_2>0$ the local maximum (with
$F''(p_2)<0$).  We denote $h_1=F(p_1)$  and $h_2=F(p_2)$, and distinguish between the following cases:

\begin{center}
  \begin{tabular}{| c | l  l  l | l  l |}

    \hline
    Case & & & & &\\
    \hline
    $(i)$ & $A>0$ &                & & $p_1<0<p_2$ & $h_1<0<h_2$   \\ \hline
    $(ii)$ & $A=0$, & $B>0$ & &    $p_1=0<p_2$ & $h_1=0<h_2$  \\ \hline
    $(iii)$ & $A<0$, & $4B^2 +6A>0$, &  $B^2+2A>0$ &   $0<p_1<p_2$ & $h_1<0<h_2$ \\  \hline
    $(iv)$ & $A<0$, & $4B^2 +6A>0$, & $B^2+2A=0$ &   $0<p_1<p_2$ & $h_1<h_2=0$ \\ \hline
    $(v)$ & $A<0$, & $4B^2 +6A>0$, & $B^2+2A<0$ &   $0<p_1<p_2$ & $h_1<h_2<0$  \\ \hline
    $(vi)$ & $A<0$, & $4B^2 +6A\leq 0$ &  & &\\   \hline
  \end{tabular}
\end{center}
Taking into account the cases described above  and  the classification given in Propositions
\ref{P_peaked} and \ref{P_cusped} we obtain the following types of singular TWS
varying the energy level $h$:

\begin{center}
  \begin{tabular}{| l || c | c | c | c | }
    \hline
         Energy/Case                & $(i)$                   & $(ii)$                  &  $(iii)$                 &  $(iv)$  \\\hline\hline
     $h<h_1$      &                          &                         &                         &        \\\hline
     $h=h_1$      &                          &                         & constant             & constant  \\\hline
     $h_1<h<0$ &                          &                         & smooth periodic  & smooth periodic \\\hline
     $h=0$         &                          & constant             & peaked periodic  & peaked solitary \\\hline
     $0<h<h_2$ & cusped periodic   & cusped periodic & cusped periodic   &        \\   \hline
     $h=h_2$     & cusped solitary     & cusped solitary  & cusped solitary    &         \\   \hline
      $h>h_2$    &                          &                         &                          &        \\ \hline
  \end{tabular}\\\vspace{1.5em}
   \begin{tabular}{| l || c |}
    \hline
           Energy/Case                        &  $(v)$  \\ \hline\hline
     $h<h_1$              &  \\ \hline
     $h=h_1$              & constant\\  \hline
     $h_1<h<h_2<0$  &  smooth periodic   \\  \hline
      $h=h_2$             & smooth solitary \\ \hline
     $h>h_2$         &  \\   \hline
  \end{tabular}
\end{center}

\subsection{Generalized Camassa-Holm Equation}
\label{S_GCH}
In \cite{Liu2004} the authors study peaked solitary and periodic cusped traveling wave solutions of a generalization of the CH equation of the form
\begin{equation}
\label{E_GCH}
  u_t + 2\kappa\,u_x -u_{txx} + a\, u\,u_x  = 2\,u_x u_{xx} + u\,u_{xxx},
\end{equation}
where $a\in\R$ is an additional parameter. At the end of their paper they state the following conjecture: ``If the parameter $a\leq 0$, then equation \eqref{E_GCH} has no peaked solitary wave solution''. Using the approach developed in the preceding sections it is easy to see that this assertion is true. Proceeding as with the CH above we introduce the traveling wave Ansatz and integrate once to find that after the change of variables $w=u-c$ we obtain
\begin{equation*}
 w''w+ \frac{(w')^2}{2} +F'(w)=0,
\end{equation*}
which is an equation of the form \eqref{E_TWeq} with
\begin{equation*}
  F(w)=  A\,w + B\,w^2 - \frac{a}{6}w^3,
\end{equation*}
where $A=r+(1 - \frac{a}{2})c^2-2\kappa c$ and
$B=(1-a)\frac{c}{2}-\kappa$. Our analysis (cf.~Proposition
\ref{P_peaked}) shows that an equation of this form has peaked
solitary TWS if and only if $F'(0)<0$ and there exists $m>0$ such
that $F(m)=F(0)$ with $F(u)<F(0)$ for $u\in(0,m)$ and $F'(m)=0$.
Assuming that $a<0$, we see that $F(w)\rightarrow \pm \infty$ as
$w\rightarrow \pm \infty$. This contradicts the conditions for the
existence of peaked solitary solutions stated above. Indeed, if
equation \eqref{E_GCH} had peaked solitary TWS, then $F'(0)=A<0$.
Thus, $F$ would have a maximum to the left and a minimum to the
right of $w=0$. Hence there exists $m>0$ such that $F(m)=F(0)$, but
$F'(m)\neq 0$ since $F$ has at most two extrema. For $a=0$ the
situation is similar. This shows that the conjecture is true.


\end{document}